\documentclass{amsart}
\usepackage[margin=1.5cm]{geometry}
\usepackage{amssymb,amsthm,comment,color,caption,enumerate,adjustbox,enumitem,tikz,
amsmath, amsfonts,amsthm,amssymb,amscd, verbatim, graphicx,color,multirow,booktabs, caption,tikz,tikz-cd, mathdots, todonotes}
\usepackage{colortbl, xcolor,bm}
\usepackage{chngcntr, arydshln, soul, float,subcaption}
\numberwithin{equation}{section}
\usetikzlibrary{arrows,shapes,positioning,decorations.markings}
\usetikzlibrary{matrix, calc, arrows, graphs}
\usepackage{hyperref}
\newtheorem{theorem}{Theorem}[section]
\newtheorem{lemma}[theorem]{Lemma}

\newtheorem{corollary}[theorem]{Corollary}

\counterwithin{table}{section}
\begin{document}
\title[Engel graphs]{On the diameter of Engel graphs}
\author[A. Lucchini]{Andrea Lucchini}
\address{Andrea Lucchini, Dipartimento di Matematica \lq\lq Tullio Levi-Civita\rq\rq,\newline
 University of Padova, Via Trieste 53, 35121 Padova, Italy} 
\email{lucchini@math.unipd.it}
         
\author[P. Spiga]{Pablo Spiga}
\address{Pablo Spiga, Dipartimento di Matematica e Applicazioni,\newline
 University of Milano-Bicocca, Via Cozzi 55, 20126 Milano, Italy} 
\email{pablo.spiga@unimib.it}
\subjclass[2010]{primary 20F99, 05C25}
\keywords{commuting graph; prime graph; Engel elements; Engel graph}        
\thanks{The authors are members of the GNSAGA INdAM research group and kindly acknowledge their support.}
	\maketitle

        \begin{abstract}
Given a finite group $G$, the Engel graph of $G$ is a directed graph $\Gamma(G)$ encoding pairs of elements satisfying some
Engel word. Namely, $\Gamma(G)$ is the directed graph, where the vertices are the non-hypercentral elements of $G$ and
where there is an arc from $x$ to $y$ if and only if $[x,_ n y] = 1$ for some
$n \in \mathbb{N}$. From previous work, it is known that, except for a few exceptions, $\Gamma(G)$ is strongly connected. In this paper, we give an absolute upper bound on the diameter of $\Gamma(G)$, when $\Gamma(G)$ is strongly connected.
        	          \end{abstract}

\section{Introduction}\label{sec:introduction}
In this paper we investigate a directed graph introduced by Peter Cameron in~\cite[Section 11.1]{cam}. Let $x$ and $y$ be free generators of a free group of rank $2$. We let $[x,y]:=x^{-1}y^{-1}xy$.  We define recursively $[x,_0y]:=x$ and $[x,_{i+1}y]:=[[x,_{i}y],y],$ for every $i\ge 0$. The word $[x,_ny]$ is the \textit{\textbf{$n^{\mathrm{th}}$ Engel word}}.  

Now, let $G$ be a group and let  $I_n(G)=\{x\in G\mid [x,_ny]=[y,_nx]=1, {\text { for every $y\in G$}}\}$ be  the set of elements of $G$ that are  right and left $n$-Engel. The $n^{\mathrm{th}}$ \textit{\textbf{Engel graph}} $$\Gamma_{n}(G)$$ of $G$ is the directed graph having vertex set $G\setminus I_n(G)$, where $(x,y)$ is declared to be an arc if and only  if $[x,_ny]=1$. Clearly, when $n:=1$, $I_1(G)$ is the center ${\bf Z}(G)$ of $G$ and $\Gamma_1(G)$ is the \textbf{\textit{commuting graph}} of $G$.

A directed graph is \textit{\textbf{strongly connected}} if, for any two vertices, there exists a directed path from the first to the second. A directed graph is \textit{\textbf{connected}} if, for any two vertices, there exists a (not necessarily  directed) path from the first to the second.   The commuting graph is undirected because the commutator word is symmetric in $x$ and $y$, but in general $\Gamma_n(G)$ is genuinely directed. The \textit{\textbf{distance}} between two vertices in a directed graph is the minimum length of a directed path from the first vertex to the second; the \textit{\textbf{diameter}} is the maximum distance between all vertices in the directed graph. 

 There is a natural reason for excluding  the elements of $I_n(G)$ from the vertex set of $\Gamma_n(G)$: if we include elements of $I_n(G)$ in the vertex set, then we define a graph which is trivially strongly connected, because any element of $I_n(G)$ is adjacent to every other vertex. Observe that, for every $n$, $\Gamma_n(G)$ is a subgraph of $\Gamma_{n+1}(G)$;  this means that the family of graphs $(\Gamma_n(G))_n$ becomes denser as $n$ increases.

 We are  interested in a ``cumulative'' version of $\Gamma_n(G)$. Let $1={\bf Z}_0(G)\le {\bf Z}_1(G)\le{\bf Z}_2(G)\le \cdots $ be the series of subgroups of $G$, where ${\bf Z}_{n+1}(G)/{\bf Z}_n(G)={\bf Z}(G/{\bf Z}_n(G))$. The subgroup $${\bf Z}_\infty(G):=\bigcup_{n\ge 0}{\bf Z}_n(G)$$ is called the \textit{\textbf{hypercenter}} of $G$. We let $\Gamma(G)$ be the directed graph having vertex set $G\setminus{\bf Z}_\infty(G)$ and where $(x,y)$ is an arc of $\Gamma(G)$ if and only if $[x,_ny]=1$, for some positive integer $n$.  It is shown  in~\cite[Introduction]{DLN} that the hypercenter ${\bf Z}_\infty(G)$ is the set of elements of $G$ with the property that, for every $y\in G$, there exists $n\in\mathbb{N}$ with $[x,_ny]=[y,_nx]=1$. As above, this is the reason for excluding  the elements of ${\bf Z}_\infty(G)$ from the vertex set of $\Gamma(G)$.

The first investigation on Engel graphs is in~\cite{DLN}. For instance,~\cite[Theorem~1.1]{DLN} shows that, if $G$ is a finite group, then $\Gamma(G)$ is weakly connected and its undirected diameter is at most 10. However, since $\Gamma(G)$ is a directed graph, it is natural to investigate the strong connectivity and the diameter of $\Gamma(G)$. The strong connectivity of $\Gamma(G)$ is intimately related to the composition factors of $G$. Corollary~1.4 in~\cite{LS} shows that, $\Gamma(G)$ is not strongly connected if and only if one of the following cases occur:
\begin{enumerate}
\item\label{type1} $G/{\bf Z}_\infty(G)$ is a Frobenius group;
\item\label{type2} $G/{\bf Z}_\infty(G)\cong\mathrm{PSL}_2(q)$ with $q\ge 4$ even or with $q\equiv 5\pmod 8$;
\item\label{type3} $G/{\bf Z}_\infty(G)\cong {}^2B_2(q)$ with $q\ge 8$;
\item\label{type4} $G/{\bf Z}_\infty(G)\cong\mathrm{Aut({}^2B_2(2^e))}$ with $e$ an odd prime.
\end{enumerate}
A slight improvement of this result, concerned with the minimum $n$ with $\Gamma_n(G)$ strongly connected, is in~\cite{DMS}.

In the light of~\cite[Corollary~1.4]{LS} and on the results on the undirected diameter in~\cite{DLN}, in this paper we are interested on the diameter of $\Gamma(G)$. 
\begin{theorem}\label{thrm:MAIN}
Let $G$ be a finite group. If $\Gamma(G)$ is strongly connected, then the diameter of $\Gamma(G)$ is at most $16$. Moreover, if $G/{\bf Z}_\infty(G)$ is not almost simple, then the diameter of  $\Gamma(G)$ is at most $12$. 
\end{theorem}
The upper bound $16$ established in our paper should not be regarded as definitive since there is a strong likelihood that it can be further improved. What is truly noteworthy is not the specific numerical value but rather the existence of such an upper bound. Indeed, rather than trying to obtain the best possible bound, our proof heavily relies on the diameter of the connected components of commuting graphs, see~\cite{mp}. Incidentally,~\cite[Theorem~1.3]{DLN} shows that, when $G$ is soluble and $\Gamma(G)$ is strongly connected (that is, $G/{\bf Z}_\infty(G)$ is not a Frobenius group), then $\Gamma(G)$ has diameter at most $4$. Furthermore, there exist soluble groups attaining the bound $4$.

\section{Notation and preliminaries}\label{not&prel}
All groups in this paper are finite, therefore, further specification is unnecessary.

 Our notation is standard, given a group $G$, we denote by ${\bf F}(G)$ the \textit{\textbf{Fitting subgroup}} of $G$ and by ${\bf F}^\ast(G)={\bf F}(G){\bf E}(G)$ the \textit{\textbf{generalized Fitting subgroup}}, where ${\bf E}(G)$ is the layer subgroup. Given $g\in G$, we denote by ${\bf o}(g)$ the order of $g$.

Given a directed graph $\Gamma$, we denote by $d_\Gamma(x,y)$ the \textit{\textbf{distance}} between the two vertices $x$ and $y$, that is, the minimum length of a directed path from $x$ to $y$. When the directed graph $\Gamma$ is clear from the context, we drop the label $\Gamma$ from $d_\Gamma$.

In what follows, we write $x\mapsto y$ to denote the arc $(x,y)$ of the Engel graph $\Gamma(G)$ and we write $x\mapsto_n y$ to denote the arc $(x,y)$ of the $n^{\mathrm{th}}$ Engel graph $\Gamma_n(G)$. 

Although we are interested in Engel graphs, our inductive proof of Theorem~\ref{thrm:MAIN} requires some auxiliary graphs. We let 
$\Lambda(G)$ be the directed graph having vertex set $G$, where $(x,y)$ is declared to be an arc if $[x,_ny]=1$, for some some $n\in\mathbb{N}$. Also, we let $\Delta(G)$ be the subgraph of $\Lambda(G)$ induced on $G\setminus\{1\}$. In particular, $\Gamma(G)$ is the subgraph of $\Lambda(G)$ induced on $G\setminus{\bf Z}_\infty(G)$. Moreover, when ${\bf Z}_\infty(G)=1$, we have $\Gamma(G)=\Delta(G)$.

Some of the results of this section are taken from~\cite{DLN}, but we have rephrased them to suit our current requirements. We do not intend to claim  originality for these results. Building upon the fundations laid by Detomi, Lucchini and Nemmi, and combining their methods with the results on almost simple groups in Section~\ref{almost simple groups}, we believe this to be the most effective strategy for proving Theorem~\ref{thrm:MAIN}.

\begin{lemma}\label{lemma21}The digraph $\Gamma(G)$ is strongly connected if and only if so is
$\Gamma(G/{\bf Z}_\infty (G))$. Moreover, $\mathrm{diam}(\Gamma(G/{\bf Z}_\infty(G)))= \mathrm{diam}(\Gamma(G))$.
\end{lemma}
\begin{proof} 
Let $\bar G=G/{\bf Z}_\infty(G)$. We adopt the ``bar'' notation to denote the natural homomorphism from $G$ onto $\bar G$.

Observe that, if  $(x,y)$ is an arc of $\Gamma(G)$, then $(\bar x,\bar y)$  is an arc of $\Gamma(G/{\bf Z}_\infty (G))=\Gamma(\bar G)$. Conversely, if $(\bar x,\bar y)$ is an arc of $\Gamma(\bar G)$, then $[\bar y,_n \bar x]=1$ for some $n\in \mathbb{N}$. Therefore, $[y,_nx]\in  {\bf Z}_\infty (G)$. As $G$ is finite,  ${\bf Z}_\infty (G)={\bf Z}_m (G)$, for some
$m\in \mathbb{N}$. Thus $[y,_{n+m} x] = [[y,_n x],_m x] = 1$ and hence $(x,y)$ is an arc of  $\Gamma(G)$.

The rest of the proof follows immediately from the previous paragraph.
\end{proof}

The relevance of Lemma~\ref{lemma21} in the proof of Theorem~\ref{thrm:MAIN} is clear: it allows to replace an arbitrary group $G$  with $G/{\bf Z}_\infty(G)$, which has the advantage of having trivial center.

\begin{lemma}\label{lemma32}Let $G$ be a group and let  $N$ be a non-nilpotent normal subgroup of $G$. If $d_{\Delta(G)}(y_1,y_2)\leq k$ for every pair $y_1, y_2$ of non-identity elements of $N$, then $\Delta(G)$ is strongly connected and 
$\mathrm{diam}(\Delta(G))\le k+4$. 
\end{lemma}
\begin{proof}
Let $x$ be a non-identity element of $G$ and let $\tilde x$ be an element of prime order in $\langle x\rangle$. If ${\bf C}_N(\tilde x)=1$, then the action of $\tilde x$ by conjugation on $N$ induces a fixed-point-free automorphism of $N$. By~\cite[Theorem~10.2.1]{10}, $N$ is nilpotent, which contradicts our hypothesis on $N$. Therefore, there exists $y\in {\bf C}_N(\tilde x)$ with $y\ne 1$. Clearly, $(x,\tilde x),(\tilde x,x)$, $(\tilde x,y)$ and $(y,\tilde x)$ are arcs of $\Delta(G)$. 

The paragraph above shows that, for every non-identity element $x$ of $G$, there exists $y\in N\setminus\{1\}$ with $$d_{\Delta(G)}(x,y),d_{\Delta(G)}(y,x)\le 2.$$ Therefore, $\mathrm{diam}(\Delta(G))\le \mathrm{diam}(\Delta(N))+4$. 
\end{proof}

\subsection{Almost simple groups}\label{almost simple groups}
The scope of this section is to investigate the diameter of Engel graphs of almost simple groups. Strictly speaking, our results follow from the work in~\cite{LS} investigating the strong connectivity of Engel graphs of almost simple groups. However, without going over all the reasoning in~\cite{LS}, we only give some detail.

We  first need a relation among the Engel graph, the \textit{\textbf{prime graph}} and
the \textit{\textbf{commuting graph}}. Given a finite group $G$, we denote by $\pi(G)$ the set of prime divisors of the order of $G$. More
generally, given a positive integer $n$, we denote by $\pi(n)$ the set of prime divisors of $n$. Now, the prime graph $\Pi(G)$ of
$G$ is the graph having vertex set $\pi(G)$ and where two distinct primes $r$ and $s$ are declared to be adjacent if and only if
$G$ contains an element having order divisible by $rs$. The commuting graph is the graph having vertex set $G \setminus {\bf Z}(G)$ and
where two distinct elements of $G$ are declared to be adjacent if they commute; as we mentioned in Section~\ref{sec:introduction}, in our current terminology, we may denote the commuting graph by $\Gamma_1 (G)$.

Suppose now that ${\bf Z}(G) = 1$. Observe that $\Gamma_1 (G)$ is a subgraph of the Engel graph $\Gamma_n (G)$, in particular, the connected
components of the commuting graph $\Gamma_1 (G)$ give useful information on the connected components of $\Gamma_n(G)$ and hence also of $\Gamma(G)$. Now, a key
result of Williams~\cite{Williams} (see also~\cite[Theorem~4.4]{mp}) gives a method to control the connected components of $\Gamma_1 (G)$ using
the connected components of the much simpler graph $\Pi(G)$.

\begin{theorem}[Theorem~4.4,~\cite{mp}]\label{thrm21} Let $G$ be a non-soluble group with ${\bf Z}(G) = 1$, let $\Psi$ be a connected component of the
commuting graph of $G$ and let $\psi$ be the set of prime divisors of the elements of $\Psi$. If $2\notin\psi$, then $G$ has an abelian Hall $\psi$-subgroup $H$ which is isolated in
the commuting graph of $G$ and $\Psi = H \setminus \{1\}$. In particular, $\psi$ is a connected component of the prime graph of $G$.
\end{theorem}

Theorem~\ref{thrm21} describes the disconnected components of the commuting graphs of non-soluble groups. They consist of
perhaps more than one connected component containing involutions and then all the remaining connected components
are complete graphs. Moreover, these remaining connected components determine (and are determined) by the connected
components of the prime graph consisting of odd primes. To gain a comprehensive understanding of the connected components within the commuting graph of a non-abelian simple group, it is essential to ascertain the distribution of even-order elements among these components. Furthermore, for odd-order elements, in view of Theorem~\ref{thrm21}, we rely on Williams' classification of the connected components of the prime graph of simple groups~\cite{Williams}.
 This approach has been used widely in~\cite{DMS,LS}. 
 
The Brauer-Fowler theorem helps to deal with even-order elements.
\begin{theorem}[Lemma~3.5,~\cite{mp}] Let $G$ be a group with trivial center and at least two conjugacy classes of involutions.
Then, the commuting graph has a unique connected component containing all the elements of even order in $G$. 
\end{theorem}

For simple groups of Lie type the work of Morgan and Parker~\cite{mp} simplifies further the analysis of almost simple groups and allow us to deduce important informations on the connected components of $\Gamma_1(G)$.
\begin{theorem}[Proposition~8.10,~\cite{mp}]\label{theorem111}
	Let $G$ be a simple group of Lie type and assume G is not isomorphic to one of
	the following:
 $$\mathrm{PSL}_3(4), \mathrm{PSL}_2(q), {}^2B_2(q), {}^2G_2(q), {}^2F_4(q) \hbox{ with }q\ge 8.$$ Then the commuting graph of $G$ has a connected component containing all elements of even order of $G$.
\end{theorem}
Combining Theorems~\ref{thrm21} and~\ref{theorem111}, we obtain this useful reduction.
\begin{corollary}[Corollary~2.4,~\cite{LS}]\label{cor}
Let $n$ be a positive integer, let $G$ be a non-abelian simple group  with the property that the commuting graph of $G$ has a connected component, $\Omega$ say, containing all elements having even order.  Then $\Gamma_n(G)$ is strongly connected if  for every connected component $\psi$ of the prime graph of $G$ with $2\notin\psi$ and for every Hall $\psi$-subgroup $H$ of $G$ (whose existence is guaranteed by Theorem~$\ref{thrm21}$), there exists $h\in H\setminus \{1\}$ and $x,y\in\Omega$ with $x\mapsto_n h$ and $h\mapsto_n y$. 
\end{corollary}

The following result is of paramount importance for the proof of Theorem~\ref{thrm:MAIN}.
\begin{theorem}[Theorem~1.1,~\cite{mp}]\label{thrm11}
Suppose that $G$ is a finite group with trivial center. Then every connected component of the
commuting graph of $G$ has diameter at most $10$.
\end{theorem}

 Following~\cite{DLN},~\cite{mp} and~\cite{Williams}, we let $\pi_1(G)$ be the connected component of the prime graph of $G$ containing the prime $2$.

\begin{lemma}\label{lemma-1}
Let $G$ be an almost simple group with $\Gamma(G)$ strongly connected. Then  $\mathrm{diam}(\Gamma(G))\le 16$.
\end{lemma}

\begin{proof}
 Let $T$ be the socle of $G$. 
 We give a complete proof only when $T$ is an alternating group. From~\cite{5} or from the first lines of the proof of~\cite[Theorem 7.1]{mp}, we see that all elements of order $2$ of $T$ belong to the same connected component, $\Omega$ say, of the commuting graph of $T$.
 
Assume first that the commuting graph of $T$ is connected, that is, $\Omega=T\setminus\{1\}$. Then, by Theorem~\ref{thrm11}, $\mathrm{diam}(\Delta(T))=\mathrm{diam}(\Gamma(T))\le\mathrm{diam}(\Gamma_1(T))\le 10$. Therefore, from Lemma~\ref{lemma32}, $\mathrm{diam}(\Gamma(G))\le \mathrm{diam}(\Gamma(T))+4\le 14$.

Assume that the commuting graph of $T$ is disconnected. From Theorem~\ref{thrm21}, the prime graph of $T$ is also disconnected. When $n\ge 8$, from~\cite[Table~I]{Williams}, there exists a prime $p$ with $n\in \{p,p+1,p+2\}$ and $\pi(T)=\pi_1(T)\cup\{p\}$. When $n=7$, the connected components of the prime graph are $\{2,3\},\{5\}$ and $\{7\}$. When $n\in \{5,6\}$, the connected components of the prime graph are $\{2\},\{3\}$ and $\{5\}$. We deal with each of these cases in turn.

When $n\le 7$, we have verified the veracity of the statement of this lemma with a computer, using the computer algebra system \texttt{magma}~\cite{magma}. Strictly speaking, this is not necessary, but it avoids lengthy ad-hoc arguments for these small groups.

Suppose $n\in\{p,p+1,p+2\}$ for some prime number $p>7$. From Theorem~\ref{thrm21}, the connected components of $\Gamma_1(G)$ are $\Omega$ and one connected component for each Sylow $p$-subgroup of $T$. 
 A computation yields
$$(1,3,5)^{-1}(1,\ldots,p)^{-1}(1,3,5)(1,\ldots,p)=(1,5,3)(2,4,6).$$
This shows that, for every element $x\in T$ having order $p$, there exists $y$ having order $3$, with $[x,_2y]=1$, that is, $(x,y)$ is an arc of $\Gamma(T)$. Conversely, since ${\bf N}_T(\langle x\rangle)$ has order $p(p-1)/2$, there exists $z\in {\bf N}_T(\langle x\rangle)$ having order $(p-1)/2$. Now, $[z,_2x]=1$ and hence $(z,x)$ is an arc of $\Gamma(T)$. 
From~\cite[Theorem~7.1]{mp}, $\Omega$ has diameter $8$ in the commuting graph and hence $\mathrm{diam}(\Gamma(T))\le 8+2=10$. Now, from Lemma~\ref{lemma32} applied with $N=T$, we deduce $\mathrm{diam}(\Gamma(G))\le 10+4=14$.

All other simple groups $T$ are dealt with similarly and all the relevant information for adapting the proof above for sporadic simple groups and for groups of Lie type is in~\cite{LS}. We explain here the main idea in~\cite{LS} in one particular case. Assume that all the even order elements in $T$ belong to the same connected component, $\Omega$ say, of $\Gamma_1(T)$. The work in~\cite{LS} shows that, for any two non-identity elements $g_1,g_2$ of $T$, there exists $z_1,z_2\in \Omega$ with $g_1\mapsto z_1$ and $z_2\mapsto g_2$. In~\cite{LS}, this was enough to guarantee that 
$\Gamma(T)$ is strongly connected. However, here, for bounding the diameter of $\Gamma(T)$ we may simply use~\cite{mp}. Indeed, we obtain
$d(g_1,g_2)\le d(g_1,z_1)+d(z_1,z_2)+d(z_2,g_2)\le 1+10+1=12$ and $\mathrm{diam}(\Gamma(T))\le 12$. Therefore, from Lemma~\ref{lemma32}, $\mathrm{diam}(\Gamma(G))\le \mathrm{diam}(\Gamma(T))+4\le 16$. The only groups where this argument does not apply are the sporadic simple groups and the Lie type groups in Theorem~\ref{theorem111}. These groups are deal with ad-hoc arguments in~\cite{LS} obtaining the same bound. 
\end{proof}
One should not take too seriously the upper bound in Lemma~\ref{lemma-1}, especially in the context of almost simple groups with an alternating group as their socle; it is worth noting that ad-hoc arguments can yield significantly more precise bounds for this particular case.

\subsection{Reduction lemmas}\label{sec:reductionlemmas}We conclude this section with a series of lemmas that are used in the proof of Theorem~\ref{thrm:MAIN} to create an inductive process.

\begin{lemma}\label{lemma22DLN}
Let $G$ be a group. If $x\in G$ and $y\in {\bf F}(G)$, then $(x,y)$ is an arc of $\Lambda$. 
\end{lemma}
\begin{proof}It follows from the fact that the Fitting subgroup ${\bf F} (G)$
is the set of left Engel elements of $G$.
\end{proof}

\begin{lemma}\label{lemma33}Let $G$ be a group and let  $X,Y$ be two non-identity subgroups of $G$ with $[X,Y]=1$ and $G=XY$. Then $\Delta(G)$ is strongly connected and $\mathrm{diam}(\Delta(G))\le 3$.
\end{lemma}
\begin{proof}
Leg $g_1=x_1y_1$ and $g_2=x_2y_2$ be two non-identity elements of $G$ with $x_1, x_2\in X$ and $y_1,y_2\in Y$. We show that $g_1$ and $g_2$ are at distance at most $3$ in $\Delta(G)$.
Suppose first $x_1$ and $y_2$ are not the identity. Then $$g_1=x_1y_1,\,\,x_1,\,\,y_2,\,\,x_2y_2=g_2$$
is a path of length $3$ in $\Gamma_1(G)$, and hence also in $\Delta(G)$. Similarly, if $x_2$ and $y_1$ are not the identity, then $$g_1=x_1y_1,\,\,y_1,\,\,x_2,\,\,x_2y_2=g_2$$
is a path of length $3$ in $\Gamma_1(G)$, and hence also in $\Delta(G)$. When $x_1= 1$ and $x_2\ne 1$,  $g_1=y_1,\,\,x_2,\,\,x_2y_2=g_2$
is a path of length $2$ in $\Gamma_1(G)$, and hence also in $\Delta(G)$. Analogously, when $x_1\ne  1$ and $x_2= 1$,  $g_1=x_1y_1,\,\,x_1,\,\,y_2=g_2$
is a path of length $2$ in $\Delta(G)$. When $x_1=x_2=1$, let $x\in X\setminus\{1\}$ and observe that  $g_1=y_1,\,\,x,\,\,y_2=g_2$
is a path of length $2$ in $\Delta(G)$. The only case that remains to discuss is when $x_1,x_2\ne 1$ and $y_1=y_2=1$. Let $y\in Y\setminus\{1\}$ and observe that  $g_1=x_1,\,\,y,\,\,x_2=g_2$
is a path of length $2$ in $\Delta(G)$. 
\end{proof}

\begin{lemma}\label{lemma34}If $G$ is a group with ${\bf F}^\ast(G)\ne {\bf F}(G)$ and with $\Delta(G)$ is strongly connected, then $\mathrm{diam}(\Delta(G))\le 16$.
\end{lemma}
\begin{proof}
Recall that ${\bf F}^\ast(G)={\bf F}(G){\bf E}(G)$ is the central product of the Fitting subgroup ${\bf F}(G)$  with the layer subgroup ${\bf E}(G)$ of $G$.

Since ${\bf E}(G)$ is the central product of its components, either ${\bf F}^\ast(G)$ is a non-abelian simple group or $\Delta({\bf F}^\ast(G))$ is strongly connected of diameter at most 3 by Lemma~\ref{lemma33}. In the latter case, Lemma~\ref{lemma32} gives that $\Delta(G)$ is strongly connected  of diameter at most $7$.

Assume ${\bf F}^\ast(G)$ is non-abelian simple, that is, $G$ is almost simple. As $\Gamma(G)$ is strongly connected, the diameter is bounded above by $16$ by Lemma~\ref{lemma-1}.
\end{proof}

\begin{lemma}[Lemma~3.5,~\cite{DLN}]\label{lemma35DLN}Let $G$ be a group and let $x$ and $y$ be two non-identity elements of $G$. Suppose ${\bf Z}_\infty (G) = 1$, ${\bf F} (G) = {\bf F}^\ast (G)$ and $G$ is not a Frobenius group.
 If $d_{\Gamma(G)}(x, y) > 4$ and $y_r$ is a power of $y$ of order a prime $r$, then ${\bf C}_G (y_r )$ has odd order and it is metacyclic. In particular, $r$ is odd.
\end{lemma}

\begin{proof}
There exist $x_1, x_2\in G$ of order 2 such that $[g_1,x_1]=[g_2,x_2]=1.$ By Lemma \ref{lemma35DLN}
$$d_{\Gamma(G)}(g_1, g_2)\leq d_{\Gamma(G)}(g_1, x_1)+d_{\Gamma(G)}(x_1, x_2)+d_{\Gamma(G)}(x_2, g_2)\leq 1+4+1=6. \qedhere$$
\end{proof}

\begin{lemma}[Proposition~3.10,~\cite{DLN}]\label{proposition3.10} Assume ${\bf Z}_\infty (G) = 1$, ${\bf F} (G) = {\bf F}^\ast (G)$ and there
exist $x, y \in G$ with $d_{\Delta(G)}(x, y) > 4$. Let $J/{\bf F} (G) = {\bf F} (G/{\bf F} (G))$ and $J^\ast /{\bf F} (G) = {\bf F}^\ast (G/{\bf F} (G))$. If $J = J^\ast$, then $G$ is a Frobenius group.
\end{lemma}

\begin{lemma}\label{technical}
Let $G$ be a group and let $x\in G$ having prime order $p$ such that ${\bf N}_G(\langle x\rangle){\bf F}(G)$ is a Frobenius group with Frobenius kernel ${\bf F}(G)$ and Frobenius complement ${\bf N}_G(\langle x\rangle)$. Then the mapping $$\theta:{\bf N}_G(\langle x\rangle)\to {\bf N}_{G/{\bf F}(G)}(\langle x{\bf F}(G)\rangle)$$ sending $g$ to $g{\bf F}(G)$ $\forall g\in G$ is surjective and the image of ${\bf C}_G(x)$ via $\theta$ is ${\bf C}_{G/{\bf F}(G)}(x{\bf F}(G))$. In particular,
$$\frac{{\bf N}_G(\langle x\rangle)}{{\bf C}_G( x)}\cong\frac{{\bf N}_{G/{\bf F}(G)}(\langle x{\bf F}(G)\rangle)}{{\bf C}_{G/{\bf F}(G)}(x{\bf F}(G))}.$$
\end{lemma}
\begin{proof}
Let $g{\bf F}(G)\in {\bf N}_{G/{\bf F}(G)}(\langle x{\bf F}(G)\rangle)$. Then $x^g=x^if_1$, for some $i\in\{1,\ldots,p-1\}$ and some $f_1\in {\bf F}(G)$. Since $x$ acts fixed-point-freely on ${\bf F}(G)$, the element $x^if_1$ has order $p$. Since $\langle x\rangle\ltimes {\bf F}(G)$ is a Frobenius group, there exists $f_2\in {\bf F}(G)$ with $x^g=(x^if_1)^{f_2}=x^i$. Thus $x^{gf_2^{-1}}=x^{i}$, $gf_2^{-1}\in {\bf N}_G(\langle x\rangle)$ and $\theta(gf_2^{-1})=g{\bf F}(G)$.

Clearly, $\theta({\bf C}_G(x))\le {\bf C}_{G/{\bf F}(G)}(x{\bf F}(G))$. Conversely, to prove the reverse inclusion, we may argue as in the paragraph above observing that $i=1$.
\end{proof}

Let $q$ be a non-zero integer  and let $t$ be a positive integer. Recall that a \textit{\textbf{primitive prime divisor}} of $q^t-1$ is a prime $p$ such that $p\mid q^t-1$ and $p\nmid q^i-1$, for all $1\le i<t$. We also write $t=\mathrm{ord}_p(q)$. Zsigmondy's theorem~\cite{zi} shows that $q^t-1$ admits a primitive prime divisor, except when $t=2$ and $q$ is a Mersenne prime, or when $(t,q)=(6,2)$.

\begin{lemma}\label{lemma:simple}
Let $S$ be a non-abelian simple group, let $p$ be an odd prime number dividing $|S|$ and let $P$ be a Sylow $p$-subgroup of $S$ with $|{\bf N}_S(P):{\bf C}_S(P)|$ odd. Then $(S,p,|{\bf N}_S(P)|)$ appears in Table~$\ref{table1}$.
\end{lemma}
\begin{proof}
Building on the work of Guralnick, Navarro and Tiep~\cite[Theorem~A]{GNT}, Xu and Zhou~\cite{XZ} give useful information on the non-abelian simple groups admitting a Sylow $p$-subgroup, for some odd prime $p$, with $|{\bf N}_S(P):{\bf C}_S(P)|$ odd. Indeed, from~\cite[Theorem~1.1]{XZ} we deduce that either $P$ is cyclic or $S\cong\mathrm{PSL}_2(q)$ with $q\equiv 3\pmod 4$. Suppose first $S\cong\mathrm{PSL}_2(q)$ with $q$ a prime power (we deal with all cases, regardless of the congruence of $q$ modulo $4$). If $p\mid q$, then $|{\bf N}_S(P):{\bf C}_S(P)|=(q-1)/\gcd(q-1,2)$, which is odd if and only if $q\equiv 3\pmod 4$. If $p\mid q\pm 1$, then $|{\bf N}_S(P):{\bf C}_S(P)|=2$ is even.  In particular, for the rest of the proof, we may suppose that $P$ is cyclic and that $S$ is not isomorphic to $\mathrm{PSL}_2(q)$.

Assume $S=\mathrm{Alt}(n)$, for some $n\ge 5$. Since $P$ is cyclic, $n<2p$ and hence $P$ is generated by a cycle of length $p$. Therefore $|{\bf N}_S(P)|=(p-1)p(n-p)!/2$ and 
$$|{\bf C}_S(P)|=
\begin{cases}
p(n-p)!/2&\textrm{if }p<n-1\\
p&\textrm{if }p\ge n-1.
\end{cases}$$ Thus $|{\bf N}_S(P):{\bf C}_S(P)|$ is odd if and only if $n\in \{p,p+1\}$ and $p\equiv 3\pmod 4$.

Assume $S$ is a sporadic simple group. Here the proof follows by using~\cite{atlas}. 

Assume that $S$ is a simple group of Lie type defined over the finite field $\mathbb{F}_q$. Since we use the results in~\cite{GNT,XZ}, we use the Lie notation for $S$. When $S$ is isomorphic to $B_2(2)'$, $G_2(2)'$, ${}^2F_4(2)'$ or ${}^2G_2(3)'$, the proof follows with a computation with the computer algebra system \texttt{magma}~\cite{magma}. Therefore, we may exclude these groups from further analysis.

 Now,~\cite[Corollary]{XZ} shows that either $p$ divides $q$, or  $S$ is of type $A_m$ or ${}^2A_ m$ with $m \ge 2$, $D_m$ or ${}^2D_m$ with $2 \nmid m$ and with $m \ge 5$, $E_6$ or ${}^2E_6$. We deal with each of these cases in turn.

Suppose first that $p$ divides $q$. By the Borel-Tits theorem, ${\bf C}_S(P)\le P$ and hence $|{\bf N}_S(P)|$ is odd, because so is $|{\bf N}_S(P):{\bf C}_S(P)|$. Therefore, we may apply~\cite[Theorem~A]{GNT} and we obtain that no example arises in this case (recall that we have dealt with $\mathrm{PSL}_2(q)=A_1(q)$ above).  For the rest of the proof we may suppose that $p$ does not divide $q$.

\smallskip

Assume $S=A_m(q)=\mathrm{PSL}_{m+1}(q)$ or $S={}^2A_m(q)=\mathrm{PSU}_{m+1}(q)$. Set $n=m+1$, and $\varepsilon=+$ when $S=\mathrm{PSL}_n(q)$ and $\varepsilon=-$ when $S=\mathrm{PSU}_n(q)$. Let $\ell=\mathrm{ord}_p(\varepsilon q)$.  If $\ell=1$, then by considering the diagonal matrices in $\mathrm{SL}_n(q)$ or in $\mathrm{SU}_n(q)$ we see that  a Sylow $p$-subgroup of $S$ cannot be cyclic, because $n\ge 3$. Therefore $\ell>1$ and hence $P$ is isomorphic to a Sylow $p$-subgroup of $\mathrm{SL}_n(q)$ or of $\mathrm{SU}_n(q)$.  We use some information from~\cite{GLS3} for the Sylow subgroup structure in cross-characteristic. Indeed, $P$ is the semidirect product of a toral part and of a Weyl part. Since $P$ is cyclic, the Weyl part must be the identity and hence $p$ does not divide the Weyl group of $\mathrm{SL}_n(q)$ or of $\mathrm{SU}_n(q)$, that is, $p>n$. Similarly, since the toral part must be cyclic, we have $2\ell> n$. Therefore a generator of $P$ is conjugate to a matrix of the form
\[
\begin{pmatrix}
A&0\\
0&I
\end{pmatrix},
\]
where $I$ is an $(n-\ell)\times(n-\ell)$-identity matrix and $A$ acts irreducibly. From this it follows that, $|{\bf N}_S(P):{\bf C}_S(P)|=\ell$. Therefore $\ell$ is odd.

\smallskip

Assume $S=D_m(q)=\mathrm{P}\Omega_{2m}^+(q)$ or $S={}^2\mathrm{P}\Omega_{2m}^-(q)$, with $m\ge 5$ odd.  Set $\varepsilon=+$ when $S=\mathrm{P}\Omega^+_{2m}(q)$ and 
$\varepsilon=-$ when $S=\mathrm{P}\Omega_{2m}^-(q)$. Since $p$ is odd, $P$ is isomorphic to a Sylow $p$-subgroup of $\Omega_{2m}^\varepsilon(q)$. As in the case above, $P$ is the semidirect product of the toral part and of the Weyl part. As $P$ is cyclic, the toral part must be the identity and hence $p>m$. Recall that
$$|\Omega_{2m}^\varepsilon(q)|=\frac{1}{2}q^{m(m-1)}(q^m-\varepsilon 1)\prod_{i=1}^{m-1}(q^{2i}-1).$$
Let $e=\mathrm{ord}_p(q)$.  If $e=m$, then $\varepsilon=+$ and $P$ is contained in a maximal non-split torus of $\Omega_{2m}^+(q)$ having order $q^{m}-1$, whose normalizer has order $(q^m-1)m$. Thus $|{\bf N}_S(P):{\bf C}_S(P)|=m$ is odd and we obtain one of the examples in Table~\ref{table1}. Similarly, if $e=2m$, then $\varepsilon=-$ and $P$ is contained in a maximal non-split torus of $\Omega_{2m}^-(q)$ having order $q^{m}+1$, whose normalizer has order $(q^m+1)m$. Thus $|{\bf N}_S(P):{\bf C}_S(P)|=m$ is odd and we obtain one of the examples in Table~\ref{table1}. Therefore, for the rest of the proof, we may suppose that $p$ divides $q^{2i}-1$, for some $i\le m-1$.

We use some classical  results on the $\Omega_{2m}^\varepsilon(q)$-conjugacy classes, we use~\cite{BurGiu} as a reference. Let $x\in P$ with ${\bf o}(x)=p$. Firstly, by~\cite[Lemma~3.5.3]{BurGiu}, two semisimple elements of odd order in $\mathrm{O}_{2m}^\varepsilon(q)$ are $\mathrm{O}_{2m}^\varepsilon$-conjugate if and only if they are conjugate in $\mathrm{GL}_n(q)$. Secondly, by~\cite[Proposition~3.5.8]{BurGiu}, we have $x^{\mathrm{P}\Omega_{2m}^\varepsilon(q)}=x^{\mathrm{PSO}_{2m}^\varepsilon(q)}$, and
 $x^{\mathrm{PSO}_{2m}^\varepsilon(q)}=x^{\mathrm{PO}_{2m}^\varepsilon(q)}$ unless $x$ acts on the underlying vector space of $S$ fixing only the zero vector. By combining these two results and by observing that $x$ does not act fixed-point-freely on the underlying vector space (because $e=\mathrm{ord}_p(x)\le 2i$, with $i<m$), we deduce that $x$ and $x^{-1}$ are $\mathrm{P}\Omega_{2m}^\varepsilon(n)$-conjugate and hence $|{\bf N}_S(P):{\bf C}_S(P)|$ is even.  

\smallskip

Assume $S=E_6(q)$ or $S={}^2E_6(q)$. The result in this case follows from~\cite[Table~C]{LSS}.
\end{proof}

\begin{corollary}\label{cor:corollary}
Let $S$ be a non-abelian simple group, let $p_1,\ldots,p_\ell$ be prime divisors of the order of $S$ with $p_i\mid p_{i+1}-1$ $\forall i\in \{1,\ldots,\ell-1\}$ and with $|{\bf N}_S(P_i):{\bf C}_S(P_i)|$ odd for all Sylow $p_i$-subgroups of $S$ and $\forall i\in \{1,\ldots,\ell\}$. Then $\ell\le 2$. 
\end{corollary}
\begin{proof}
This follows from Lemma~\ref{lemma:simple}. 
\end{proof}

\begin{table}[ht]
\begin{tabular}{c|c|c|c|c}\toprule[1.5pt]
Type&Group&prime&$|{\bf N}_S(P):{\bf C}_S(P)|$&Comments\\
\midrule[1.5pt]
alternating&$\mathrm{Alt}(n)$&$p$&$\frac{p-1}{2}$&$n\in \{p,p+1\}$, $p\equiv 3\pmod 4$\\
linear&$\mathrm{PSL}_2(q)$&$p\mid q$&$\frac{q-1}{2}$&$q\equiv 3\pmod 4$\\
&$\mathrm{PSL}_n(q)$&$p$&$\ell$&$p$ p.p.d. of $q^\ell-1$, $p>n$, $2\ell>n$, $\ell$ odd\\
unitary&$\mathrm{PSU}_n(q)$&$p$&$\ell$&$p$ p.p.d. of $(-q)^\ell-1$, $p>n$, $2\ell>n$, $\ell$ odd\\
orthogonal&$\mathrm{P}\Omega_{2m}^+(q)$&$p$&$m$&$p$ p.p.d. of $q^m-1$, $m$ odd\\
&$\mathrm{P}\Omega_{2m}^-(q)$&$p$&$m$&$p$ p.p.d. of $q^m+1$, $m$ odd\\
exceptional&$E_{6}^+(q)$&$p$&$9$&$p$ p.p.d. of $q^9-1$\\
&${}^2E_{6}(q)$&$p$&$9$&$p$ p.p.d. of $q^9+1$\\
sporadic&$M_{11}$&11&$ 5$&\\
&$M_{12}$&11&$5$&\\
&&11&$5$&\\
&&7&$3$&\\
&$M_{23}$&23&$11$&\\
&&11&$5$&\\
&&7&$3$&\\
&$M_{24}$&23&$11$&\\
&&7&$3$&\\
&$HS$&11&$5$&\\
&$Co_{2}$&23&$11$&\\
&$Co_{3}$&23&$11$&\\
&&11&$5$&\\
&$O'N$&31&$15$&\\
&$J_{3}$&19&$9$&\\
&$Co_{1}$&23&$11$&\\
&$Th$&31&$15$&\\
&$Fi_{23}$&23&$11$&\\
&$HN$&19&$9$&\\
&$Fi_{24}'$&23&$11$&\\
&$B$&23&$11$&\\
&&31&$15$&\\
&&47&$23$&\\
&$M$&23&$11$&\\
&&31&$15$&\\
&&47&$23$&\\
&&59&$29$&\\
&&71&$35$&\\
\bottomrule[1.5pt]
\end{tabular}
\caption{Exceptional cases arising the Lemma~$\ref{lemma:simple}$. }
\label{table1}
\end{table}

\begin{lemma}\label{lemma:lemma37}Let $G$ be a group with  ${\bf F}^\ast (G) = {\bf F} (G)$
and with $G/{\bf F} (G) \cong S_1\times \cdots \times S_\ell$, for some non-abelian simple groups $S_1,\ldots,S_\ell$ and some positive integer $\ell$.
Then $\Delta(G)$ is strongly connected of diameter at most $8$.
\end{lemma}
\begin{proof}
Recall that the vertex set of $\Delta(G)$ is $G\setminus\{1\}$.

When ${\bf Z}_\infty(G)\ne 1$, the statement is clear because we can reach any two vertices in $\Delta(G)$ with at most two steps by pivoting on the non-identity elements of ${\bf Z}_\infty(G)$. Therefore, for the rest of the proof, we assume ${\bf Z}_\infty (G) = 1$.

From the Brauer-Suzuki theorem, there is no simple
group with generalized quaternion Sylow $2$-subgroups. In particular, by the structure of Frobenius complements in Frobenius groups, the quotient $G/{\bf F} (G)$ cannot be isomorphic to a Frobenius complement. Therefore, $G$ is not a Frobenius group. 

By Lemma~\ref{lemma22DLN}, we have 
\begin{equation}\label{eq:eq-2}
d_{\Delta(G)}(g,f)\le 1,\quad\forall g\in G\setminus\{1\} \hbox{ and }\forall f\in {\bf F}(G)\setminus\{1\}.
\end{equation}

We claim that, for any prime divisor $p$ of $|G|$, we have
\begin{equation}\label{eq:eq-3}
d_{\Delta(G)}(f,g)\le 6,\quad\forall g\in \{x\in G\mid {\bf o}(x) =p\} \hbox{ and }\forall f\in {\bf F}(G)\setminus\{1\}.
\end{equation}
We first assume $p=2$.  By Lemma~\ref{lemma35DLN}, 
\begin{equation}\label{eq:eq1}
d_{\Delta(G)}(f,g)\le 4,\quad \forall g\in \{x\in G\mid {\bf o}(x) =2\} \hbox{ and }\forall f\in G\setminus\{1\}.
\end{equation}
In particular,~\eqref{eq:eq1} implies a strong form of~\eqref{eq:eq-3} when $p=2$ because here $f$ is an arbitrary non-identity element of $G$.

Assume now that $p$ is odd. Let $P$ be a Sylow $p$-subgroup of $G$. Suppose that $P$ does not act fixed-point-freely by conjugation on ${\bf F}(G)$. Therefore, there exists $s\in P$ centralizing a non-identity element $f'\in {\bf F}(G)$. Let $z\in {\bf Z}(P)$ with $z\ne 1$ and let $x$ be an arbitrary non-identity element of $P$. Since $[f',s]=[s,z]=[z,x]=1$, we get
$d_{\Delta(G)}(f',x)\le 3$. Thus, for every $f\in {\bf F}(G)$, we have 
\begin{equation}\label{eq:eqfpf}d_{\Delta(G)}(f,x)\le d_{\Delta(G)}(f,f')+d_{\Delta(G)}(f',x)\le 1+3=4
\end{equation} and~\eqref{eq:eq-3} holds true in this case. Therefore, for the rest of the proof of~\eqref{eq:eq-3}, we suppose that $P$ does act fixed-point-freely by conjugation on ${\bf F}(G)$. Observe that, when $\ell>1$, a  Sylow $p$-subgroup of $S_1\times\cdots \times S_\ell$
cannot act fixed-point-freely on ${\bf F} (G)$ and hence we have $\ell=1$. Set $S=S_1$. Observe also that $P$ is isomorphic to a Sylow $p$-subgroup of $S$ because ${\bf F}(G)\cap S=1$.

Here, to conclude the proof of~\eqref{eq:eq-3} in this remaining case, we argue by contradiction and we suppose the claim to be wrong. In particular, we choose the smallest prime $p$ witnessing that incorrectness of our claim.  
From~\eqref{eq:eq1}, $p$ is odd. Since
$P$ acts fixed-point-freely on ${\bf F} (G)$, $P$ is cyclic, generated by $u$ say. Thus $|P|={\bf o}(u)=p^n$, for some $n\in\mathbb{N}$. Observe that Lemma~\ref{lemma35DLN} implies that ${\bf N}_G(P){\bf F}(G)$ is a Frobenius group with Frobenius kernel ${\bf F}(G)$ and Frobenius complement ${\bf N}_G(P).$  As $S$ is not $p$-nilpotent, by Burnside's
Theorem~\cite[Theorem~10.1.8]{13}, we have
$P \nleq {\bf Z}({\bf N}_G (P ))$. Therefore, there exists an element $v \in {\bf N}_G (P ) \setminus {\bf C}_G (P )$. In
particular, $v{\bf C}_G (P )$ is a non-identity element of ${\bf N}_G (P )/{\bf C}_G (P )$. Replacing $v$ by a suitable power, we may suppose that $v{\bf C}_G(P)$ has prime order $q$ and that $v$ is a $q$-element. Now, ${\bf N}_G (P )/{\bf C}_G (P )$ is isomorphic to a subgroup  of the automorphism group $ \mathrm{Aut}(P ).$ Since $P$ is cyclic of order $p^n$, $\mathrm{Aut}(P)$ has order $\varphi(p^n)=p^{n-1}(p-1)$, where $\varphi$ is Euler's totient function.
Therefore, $q={\bf o}(v{\bf C}_S (P ))$  divides $p^{ n-1} (p- 1)$ and is coprime to $p$, because $P \le {\bf C}_G (P )$. Therefore, $q$ divides $p-1$ and, for every $f\in {\bf F}(G)\setminus\{1\}$, we have
$$d_{\Delta(G)}(f,u)\le d_{\Delta(G)}(f,v)+d_{\Delta(G)}(v,u)=d_{\Delta(G)}(f,v)+1.$$
If $q=2$, then~\eqref{eq:eq1} gives $d_{\Delta(G)}(f,u)\leq 5$ and hence~\eqref{eq:eq-3} follows in this case. Therefore, we may suppose that ${\bf N}_G(P)/{\bf C}_G(P)$ has odd order. From Lemma~\ref{technical}, this implies that ${\bf N}_S(\tilde{P})/{\bf C}_S(\tilde P)$ has odd order, where $\tilde P$ is a Sylow $p$-subgroup of the simple group $S$.
Similarly, if a Sylow $q$-subgroup $Q$ of $G$ does not act fixed-point-freely on ${\bf F}(G)$, then~\eqref{eq:eqfpf} gives $d_{\Delta(G)}(f,u)\leq 5$ and hence~\eqref{eq:eq-3} follows also in this case. Therefore, we may suppose that $Q$ acts fixed-point-freely on ${\bf F}(G)$. Therefore, we may apply the argument in this paragraph, with the prime $p$ replaced by $q$. By iterating the argument, we deduce that either $d_{\Delta(G)}(f,u)\le 6$\footnote{Here the upper bound passes from $5$ to $6$, because by iterating the argument we have potentially increased by 1 the path from $f$ to $u$.} $\forall f\in {\bf F}(G)\setminus\{1\}$, or there exists an odd prime $r$ and a Sylow $r$-subgroup of $R$ of $S$ with $r\mid q-1$ and with $|{\bf N}_S(R)/{\bf C}_S(R)|$ odd. Corollary~\ref{cor:corollary} guarantees that the second alternative for $R$ cannot arise,  which yields~\eqref{eq:eq-3}. This concludes the proof of our claim~\eqref{eq:eq-3}.

Now, let $g_1,g_2\in G\setminus\{1\}$ and let $p$ be an arbitrary prime dividing ${\bf o}(g_2)$ and let $h=g_2^{{\bf o}(g_2)/p}$. Let $f\in {\bf F}(G)\setminus\{1\}$. Then, by~\eqref{eq:eq-2} and~\eqref{eq:eq-3}, we get
$$d_{\Delta(G)}(g_1,g_2)\le d_{\Delta(G)}(g_1,f)+d_{\Delta(G)}(f,g_2)\le 1+d_{\Delta(G)}(f,g_2)\le 1+d_{\Delta(G)}(f,h)+d_{\Delta(G)}(h,g_2)\le 1+6+1=8.\qedhere$$
\end{proof}

\section{Proof of Theorem~$\ref{thrm:MAIN}$}\label{sec:proof}
Let $G$ be a group and suppose that $\Gamma(G)$ is strongly connected. By~\cite[Corollary~1.4]{LS},  $G/{\bf Z}_\infty(G)$ is not isomorphic to one of groups appearing in~\eqref{type1}--\eqref{type4}.  By Lemma~\ref{lemma21}, we may
assume that ${\bf Z}_\infty (G) = 1$.  By Lemma~\ref{lemma34}, we may suppose ${\bf F}^\ast  (G) = {\bf F} (G)$.

Let $J$ and $J^\ast$ be the subgroups of $G$ containing ${\bf F}(G)$ and with
$$\frac{J}{{\bf F}(G)}={\bf F}\left(\frac{G}{{\bf F}(G)}\right)\hbox{ and }
\frac{J^\ast}{{\bf F}(G)}={\bf F}^\ast\left(\frac{G}{{\bf F}(G)}\right).
$$
From Lemma~\ref{proposition3.10}, we may suppose $J\ne J^\ast$.

 If $J = {\bf F} (G)$,
then $\Delta(J^\ast )$ is strongly connected of diameter at most $8$ by Lemma~\ref{lemma:lemma37}. Therefore, in this case, $\Gamma(G)$ is strongly
connected of diameter at most $12$ by Lemma~\ref{lemma32}. Therefore, for the rest of the proof, we may suppose that $J>{\bf F}(G)$. In particular, ${\bf F}(J)={\bf F}(G)$ and $J$ is not nilpotent.

If $J$ is not a Frobenius group, then by Lemma~\ref{proposition3.10} $\Gamma(J)$ is strongly connected of diameter at most $4$ and hence, by Lemma~\ref{lemma32}, $\Gamma(G)$ is strongly connected of diameter at most $8$.  Therefore, for the rest of the proof, we may suppose that $J$ is a Frobenius group with Frobenius kernel ${\bf F}(G)$.

 Let $P/{\bf F}(G)$ be a Sylow 2-subgroup of $J^\ast /{\bf F}(G)$ 
and notice that $G$ cannot be almost simple. 
To conclude the proof, we distinguish two cases, depending on whether the group $P$ is Frobenius  or not.

If $P$ is not Frobenius, then it follows from ~\cite[Theorem~1.3]{DLN}
that $\Delta(P)$ is strongly connected of diameter at most $4$. In particular $d_{\Gamma(G)}(x_1, x_2)\leq 4$ for every pair  $x_1, x_2$
of non-identity elements of $J$. 
By Lemma~\ref{lemma32}, $\Gamma(G)$ is strongly connected of diameter at most $8$. 

Finally assume that $P$ is a Frobenius group. As ${\bf F}(P)={\bf F}(G),$ $P$ is a Frobenius group with Frobenius kernel ${\bf F}(G)$. This implies in particular that $2$ does not divide $|{\bf F}(G)|$. Thus, ${\bf F} (G)$ has a complement $L$ in
$P$ and any element of $L$ has centralizer of even order, since $P/{\bf F} (G)$ is a
central product of a Sylow $2$-subgroup of ${\bf E}(G/{\bf F} (G))$ and a Sylow $2$-subgroup of $J/{\bf F} (G)$.  In particular if  $y\in J\setminus {\bf F}(G)$, then ${\bf C}_G(y)$ has even order and, by Lemma~\ref{lemma35DLN}, $d_{\Gamma(G)}(x, y)\leq 4$ for every $1\neq x\in J.$ On the other hand, by Lemma~\ref{lemma22DLN}, 
if  $x\in J\setminus\{1\}$ and $f\in {\bf F}(G)\setminus\{1\}$, then $(x,y)$ is an arc of $\Delta(G)$. We have so proved that $d_{\Gamma(G)}(x_1, x_2)\leq 4$ for every pair  $x_1, x_2$
of non-identity elements of $J$, but then, as in the previous case,  
we may deduce from Lemma~\ref{lemma32} that $\Gamma(G)$ is strongly connected of diameter at most $8$.

\end{document}